\documentclass[12pt,fleqn,draft]{article} 
\usepackage{amsmath,amssymb,amsthm,amsfonts,bm}
\usepackage{enumerate}
\usepackage{indentfirst}
\usepackage{cite}
\usepackage[usenames,dvipsnames]{color}

\makeatletter
\def\@cite#1#2{[{{\bfseries #1}\if@tempswa , #2\fi}]}
\renewcommand{\section}{%
\@startsection{section}{1}{\z@}
{0.5truecm plus -1ex minus -.2ex}%
{1.0ex plus .2ex}{\bfseries\large}}
\def\@seccntformat#1{\csname the#1\endcsname.\ }
\makeatother

\numberwithin{equation}{section} 
\theoremstyle{theorem}
\newtheorem{thm}{Theorem}[section]

\newtheorem{lem}[thm]{Lemma}

\theoremstyle{definition}
\newtheorem{df}{Definition}[section]
\newtheorem{remark}{Remark}[section]

\newtheorem*{prth1.1}{Proof of Theorem 1.1}
\newtheorem*{prth2.1}{Proof of Theorem 2.1}
\newtheorem*{prth2.2}{Proof of Theorem 2.2}

\let\widehat\widehat

\def\Pi{\widehat\pi}

\begin{document}
\begin{center}
    \Large{{\bf Existence for a nonlocal Penrose--Fife type \\ 
                    phase field system with inertial term}}
\end{center}
\vspace{5pt}
\begin{center}
    Shunsuke Kurima%
    \\
    \vspace{2pt}
    Department of Mathematics, 
    Tokyo University of Science\\
    1-3, Kagurazaka, Shinjuku-ku, Tokyo 162-8601, Japan\\
    {\tt shunsuke.kurima@gmail.com}\\
    \vspace{2pt}
\end{center}
\begin{center}    
    \small \today
\end{center}

\vspace{2pt}
\newenvironment{summary}
{\vspace{.5\baselineskip}\begin{list}{}{%
     \setlength{\baselineskip}{0.85\baselineskip}
     \setlength{\topsep}{0pt}
     \setlength{\leftmargin}{12mm}
     \setlength{\rightmargin}{12mm}
     \setlength{\listparindent}{0mm}
     \setlength{\itemindent}{\listparindent}
     \setlength{\parsep}{0pt}
     \item\relax}}{\end{list}\vspace{.5\baselineskip}}
\begin{summary}
{\footnotesize {\bf Abstract.} 
This article deals with a nonlocal Penrose-Fife type phase field system with inertial term. 
We do not know whether we can prove existence of solutions in reference to 
Colli--Grasselli--Ito [Electron. J. Differential Equations 2002, No. 100, 32 pp.] or not 
(see Remark \ref{remaboutapp}). 
In this paper we introduce a time discretization scheme (see Section \ref{Sec2}), 
pass to the limit as the time step $h$ goes to $0$  
and obtain an error estimate 
for the difference between continuous and discrete solutions 
(see Section \ref{Sec5}). 
}
\end{summary}
\vspace{10pt}

\newpage

\section{Introduction}\label{Sec1}

Colli--Grasselli--Ito \cite{CGI2002} 
have derived existence of solutions to   
the parabolic hyperbolic Penrose--Fife phase field system  
\begin{equation}\label{P1}\tag{P1}
\begin{cases}
\left(-\frac{1}{u}\right)_{t} + (\lambda(\varphi))_{t} - \Delta u = f  
&\mbox{in}\ \Omega \times (0, T), 
\\[1mm] 
\varphi_{tt} + \varphi_t - \Delta\varphi 
                                           + \beta(\varphi) + \pi(\varphi) = \lambda'(\varphi)u
&\mbox{in}\ \Omega \times (0, T), 
\\[1mm]
\partial_\nu u + u = g 
&\mbox{on}\ \partial\Omega \times (0, T), 
\\[1mm] 
\left(-\frac{1}{u}\right)(0) = -\frac{1}{u_{0}},\ 
\varphi(0) = \varphi_0,\ \varphi_{t}(0) = v_0 
&\mbox{in}\ \Omega,      
\end{cases}
\end{equation}
where $\Omega \subset \mathbb{R}^d$ ($d= 1, 2, 3$) is a bounded domain 
with smooth boundary $\partial\Omega$, 
$T > 0$, 
$\lambda : \mathbb{R}\to\mathbb{R}$ 
is a smooth function which may have quadratic growth, 
$\beta : \mathbb{R}\to\mathbb{R}$ is a maximal monotone function, 
$\pi : \mathbb{R}\to\mathbb{R}$ is an anti-monotone function, 
$\partial_\nu$ denotes differentiation with respect to
the outward normal of $\partial\Omega$, 
$u_{0} : \Omega \to \mathbb{R}$, 
$\varphi_{0} : \Omega \to \mathbb{R}$ and 
$v_{0} : \Omega \to \mathbb{R}$   
are given functions. 
Moreover, in the case that $\lambda(\varphi) = \varphi$, 
they have proved uniqueness of solutions to \eqref{P1}. 
Assuming that $|\beta(r)| \leq c_{1}|r|^3 + c_{2}$ for all $r \in \mathbb{R}$, 
where $c_{1}, c_{2} > 0$ are some constants, 
we can obtain an estimate for $\beta(\varphi)$ 
by establishing the $L^{\infty}(0, T; H^1(\Omega))$-estimate for $\varphi$ 
and by the continuity of the embedding $H^1(\Omega) \hookrightarrow L^6(\Omega)$.

\bigskip

Existence of solutions to the singular nonlocal phase field system with inertial term
\begin{equation}\label{P2}\tag{P2}
\begin{cases}
(\ln u)_{t} + \varphi_{t} - \Delta u = f  
&\mbox{in}\ \Omega \times (0, T), 
\\[1mm] 
\varphi_{tt} + \varphi_t + a(\cdot)\varphi - J\ast\varphi 
                                           + \beta(\varphi) + \pi(\varphi) = u
&\mbox{in}\ \Omega \times (0, T), 
\\[1mm]
\partial_\nu u = 0 
&\mbox{on}\ \partial\Omega \times (0, T), 
\\[1mm] 
(\ln u)(0) = \ln u_{0},\ 
\varphi(0) = \varphi_0,\ \varphi_{t}(0) = v_0 
&\mbox{in}\ \Omega     
\end{cases}
\end{equation}
has been studied (\cite{K8}), 
where $J : \mathbb{R}^d \to \mathbb{R}$ is an interaction kernel, 
$a(x) := \int_{\Omega} J(x-y)\,dy$ and  
$(J\ast\varphi)(x) := \int_{\Omega} J(x-y)\varphi(y)\,dy$ 
for $x \in \Omega$.  
To derive the $L^{\infty}(0, T; H^2(\Omega))$-estimate for 
$\int_{0}^{t} u(s)\,ds$
is a key to establish an estimate for $\beta(\varphi)$. 
Indeed, it holds that 
\[
\frac{1}{2}|\varphi(x, t)|^2 
= \frac{1}{2}|\varphi_{0}(x)|^2 
   + \int_{0}^{t}\varphi_{t}(x, s)\varphi(x, s)\,ds
\]
and 
\begin{align*}
&\frac{1}{2}|\varphi_{t}(x, t)|^2 + \int_{0}^{t}|\varphi_{t}(x, s)|^2\,ds 
+ \widehat{\beta}(\varphi(x, t)) 
\notag \\[2mm] 
&= \int_{0}^{t}u(x, s)\varphi_{t}(x, s)\,ds 
+ \frac{1}{2}|v_{0}(x)|^2 + \widehat{\beta}(\varphi_{0}(x)) 
\notag \\ 
&\,\quad - \int_{0}^{t}(a(x)\varphi(x, s) - (J\ast\varphi(s))(x))\varphi_{t}(x, s)\,ds,  
\end{align*}
where $\widehat{\beta}(r) = \int_{0}^{r}\beta(s)\,ds$. 
Moreover, since $u >0$, we see that 
\[
\int_{0}^{t}u(x, s)\varphi_{t}(x, s)\,ds  
\leq \|\varphi_{t}\|_{L^{\infty}(\Omega\times(0, T))}\int_{0}^{t}u(x, s)\,ds.   
\]
Thus,  deriving the $L^{\infty}(0, T; H^2(\Omega))$-estimate for 
$\int_{0}^{t}u(x, s)\,ds$ from the first equation in \eqref{P2}, 
using the continuity of 
the embedding $H^2(\Omega) \hookrightarrow L^{\infty}(\Omega)$, 
applying the Young inequality and the Gronwall lemma, 
we can establish the $L^{\infty}(\Omega\times(0, T))$-estimates 
for $\varphi_{t}$ and $\varphi$, 
whence we can obtain the $L^{\infty}(\Omega\times(0, T))$-estimate 
for $\beta(\varphi)$ by assuming that $\beta$ is continuous.

\bigskip

In this paper we deal with the 
nonlocal Penrose--Fife phase field system with inertial term
\begin{equation}\label{P}\tag{P}
\begin{cases}
\left(-\frac{1}{u}\right)_{t} + \varphi_{t} - \Delta u = f  
&\mbox{in}\ \Omega \times (0, T), 
\\[1mm] 
\varphi_{tt} + \varphi_t + a(\cdot)\varphi - J\ast\varphi 
                                           + \beta(\varphi) + \pi(\varphi) = u
&\mbox{in}\ \Omega \times (0, T), 
\\[1mm]
\partial_\nu u + u = g 
&\mbox{on}\ \partial\Omega \times (0, T), 
\\[1mm] 
\left(-\frac{1}{u}\right)(0) = -\frac{1}{u_{0}},\ 
\varphi(0) = \varphi_0,\ \varphi_{t}(0) = v_0 
&\mbox{in}\ \Omega,     
\end{cases}
\end{equation}
where $\Omega \subset \mathbb{R}^d$ ($d= 1, 2, 3$) is a bounded domain 
with smooth boundary $\partial\Omega$. 
Moreover,  
we assume the four conditions:  
\begin{enumerate} 
\setlength{\itemsep}{0mm}
\item[{\bf A1}] 
$J(-x) = J(x)$ for all $x \in \mathbb{R}^d$ 
and $\displaystyle\sup_{x \in \Omega} \int_{\Omega} |J(x-y)|\,dy < + \infty$. 
\item[{\bf A2}] $\beta : \mathbb{R} \to \mathbb{R}$                                
is a single-valued maximal monotone function 
such that 
there exists a proper lower semicontinuous convex function 
$\widehat{\beta} : \mathbb{R} \to [0, +\infty)$ 
satisfying that   
$\widehat{\beta}(0) = 0$ and 
$\beta = \partial\widehat{\beta}$, 
where $\partial\widehat{\beta}$  
is the subdifferential of $\widehat{\beta}$. 
Moreover, $\beta : \mathbb{R} \to \mathbb{R}$ is local Lipschitz continuous. 
\item[{\bf A3}] $\pi : \mathbb{R} \to \mathbb{R}$ is a Lipschitz continuous function. 
\item[{\bf A4}] $f \in L^2(\Omega\times(0, T))$, 
$g \in L^2(0, T; H^{1/2}(\partial\Omega))$, 
$g \leq 0$ a.e.\ on $\partial\Omega\times(0, T)$,   
$\theta_0 := -\frac{1}{u_{0}} \in L^{2}(\Omega)$, 
$\theta_{0} > 0$ a.e.\ in $\Omega$, 
$\ln \theta_{0} \in L^1(\Omega)$,    
$\varphi_0, v_0 \in L^{\infty}(\Omega)$. 
\end{enumerate}

\medskip

We define weak solutions of \eqref{P} as follows.
%
%
%
 \begin{df}         
 A pair $(u, \varphi)$ with 
    \begin{align*}
    &u \in L^2(0, T; H^1(\Omega)),\ 
      - \frac{1}{u} \in H^1(0, T; {(H^1(\Omega))}^{*}) \cap L^{\infty}(0, T; L^2(\Omega)), \\ 
    &\varphi \in W^{2, 2}(0, T; L^2(\Omega)) \cap W^{1, \infty}(0, T; L^{\infty}(\Omega)) 
    \end{align*}
 is called a {\it weak solution} of \eqref{P} 
 if $(u, \varphi)$ satisfies 
    \begin{align*}
        &\left\langle 
          \left(- \frac{1}{u} \right)_{t}, w 
          \right\rangle_{({H^1(\Omega))}^{*}, H^1(\Omega)}
           + (\varphi_t, w)_{L^2(\Omega)} + \int_{\Omega} \nabla u \cdot \nabla w 
           + \int_{\partial\Omega} (u - g)w \notag \\[2mm] 
       &\hspace{60mm} = (f, w)_{L^2(\Omega)} 
                 \quad  \mbox{a.e.\ in}\ (0, T) 
                           \ \  \mbox{for all}\ w \in H^1(\Omega),  
     \\[5mm]
        &\varphi_{tt} + \varphi_t +a(\cdot)\varphi - J\ast\varphi 
                                                             + \beta(\varphi) + \pi(\varphi) = u 
                               \quad \mbox{a.e.\ in}\ \Omega\times(0, T), 
     \\[2mm]
        &\left(- \frac{1}{u} \right)(0) = \theta_0,\ 
          \varphi(0) = \varphi_0,\ \varphi_{t}(0) = v_0 
                                                       \quad \mbox{a.e.\ in}\ \Omega. 
     \end{align*}
 \end{df}

\medskip

The following theorem asserts existence and uniqueness of weak solutions to \eqref{P}.   
\begin{thm}\label{maintheorem1}
Assume that {\rm {\bf A1}-{\bf A4}} hold. 
Then there exists a unique weak solution $(u, \varphi)$ of \eqref{P}. 
\end{thm}

\begin{remark}\label{remaboutapp}
Even if in reference to \cite{CGI2002} we consider the approximation 
\begin{equation*}\tag*{(P)$_{N}$}\label{PN}
\begin{cases}
(\mu_{N}u_{N} + \rho_{N}(u_{N}))_{t} + (\varphi_{N})_{t} - \Delta u_{N} = f  
&\mbox{in}\ \Omega \times (0, T), 
\\[1mm] 
(\varphi_{N})_{tt} + (\varphi_{N})_t + a(\cdot)\varphi_{N} - J\ast\varphi_{N} 
\\ 
\hspace{30mm} +\ \beta_{N}(\varphi_{N}) + \pi(\varphi_{N}) = - (\rho_{N}(u_{N}))^{-1}
&\mbox{in}\ \Omega \times (0, T), 
\\[2.5mm]
\partial_\nu u_{N} + u_{N} = g 
&\mbox{on}\ \partial\Omega \times (0, T), 
\\[1mm] 
(u_{N})(0) = - (\rho_{N}(u_{0}))^{-1},\ 
\varphi_{N}(0) = \varphi_0,\ (\varphi_{N})_{t}(0) = v_0 
&\mbox{in}\ \Omega,     
\end{cases}
\end{equation*}
we do not know whether we can establish a priori estimates for \ref{PN} or not. 
Here $N \in \mathbb{N}$, 
$\mu_{N} := \frac{1}{1 + N^2}$, 
the function $\rho_{N} : \mathbb{R} \to \mathbb{R}$ is defined by 
\[
\rho_{N}(r) := 
\begin{cases}
\frac{1}{N + 1} &\mbox{if}\ r < - (N + 1), \\[1mm] 
- \frac{1}{r}     &\mbox{if}\ - (N + 1) \leq r \leq - \frac{1}{N + 1}, \\[1mm] 
N + 1             &\mbox{if}\ - \frac{1}{N + 1} < r,  
\end{cases}
\]
and the function $\beta_{N} : \mathbb{R} \to \mathbb{R}$ is defined by 
\[
\beta_{N}(r) := 
\begin{cases}
- N   &\mbox{if}\ \beta(r) \leq -N, \\[1mm] 
\beta(r)  &\mbox{if}\ -N < \beta(r) < N, \\[1mm]  
N  &\mbox{if}\ N \leq \beta(r).  
\end{cases}
\]
Although we can obtain that   
\[
\frac{1}{2}|\varphi_{N}(x, t)|^2 
= \frac{1}{2}|\varphi_{0}(x)|^2 
   + \int_{0}^{t}(\varphi_{N})_{t}(x, s)\varphi_{N}(x, s)\,ds
\] 
and 
\begin{align*}
&\frac{1}{2}|(\varphi_{N})_{t}(x, t)|^2 
+ \int_{0}^{t} |(\varphi_{N})_{t}(x, s)|^2\,ds 
+ \widehat{\beta}_{N}(\varphi_{N}(x, t)) 
\notag \\
&= \int_{0}^{t}(\rho_{N}(u_{N}(x, s)))^{-1}(-(\varphi_{N})_{t}(x, s))\,ds + \cdots,    
\end{align*}
where $\widehat{\beta}_{N}(r) = \int_{0}^{r}\beta_{N}(s)\,ds$, 
we do not know whether 
the $L^{\infty}(\Omega\times(0, T))$-estimate for 
$\left\{\int_{0}^{t}(\rho_{N}(u_{N}(x, s)))^{-1}\,ds\right\}_{N}$ 
can be derived or not, 
and then we do not know whether 
the $L^{\infty}(\Omega\times(0, T))$-estimates for 
$\{(\varphi_{N})_{t}\}_{N}$, $\{\varphi_{N}\}_{N}$ and $\{\beta(\varphi_{N})\}_{N}$ 
can be obtained or not. 
Even if we replace $- (\rho_{N}(u_{N}))^{-1}$ 
with $u_{N}$ in \ref{PN}, 
since the inequality $-u_{N} \geq 0$ does not hold, 
we see that 
\[
\int_{0}^{t}(-u_{N}(x, s))(-(\varphi_{N})_{t}(x, s))\,ds 
\nleq \|-(\varphi_{N})_{t}\|_{L^{\infty}(\Omega\times(0, T))}
\int_{0}^{t}(-u_{N}(x, s))\,ds,  
\]
whence we do not know whether 
the $L^{\infty}(\Omega\times(0, T))$-estimates for 
$\{(\varphi_{N})_{t}\}_{N}$, $\{\varphi_{N}\}_{N}$ and $\{\beta(\varphi_{N})\}_{N}$ 
can be established or not. 
\end{remark}

\bigskip

This paper is organized as follows. 
In Section \ref{Sec2} we introduce a time discretization of \eqref{P} 
and set precisely the approximate problem. 
In Section \ref{Sec3} 
we prove existence for the discrete problem. 
In Section \ref{Sec4}
we establish uniform estimates for the approximate problem.  
Section \ref{Sec5} obtains 
Cauchy's criterion for solutions of the approximate problem
and 
is devoted to the proofs of 
existence and uniqueness of weak solutions to \eqref{P} 
and an error estimate between solutions of \eqref{P} and solutions of the approximate 
problem.

\vspace{10pt}

\section{Time discretization}\label{Sec2}

To prove existence of weak solutions to \eqref{P}  
we deal with the discrete problem 
\begin{equation*}\tag*{(P)$_{n}$}\label{Pn}
     \begin{cases}
         \frac{\theta_{n+1} - \theta_{n}}{h} + \frac{\varphi_{n+1}-\varphi_{n}}{h}
         - \Delta u_{n+1} = f_{n+1}   
         & \mbox{in}\ \Omega, 
 \\[2mm]
         z_{n+1} + v_{n+1} + a(\cdot)\varphi_{n} - J\ast\varphi_{n} 
         + \beta(\varphi_{n+1}) + \pi(\varphi_{n+1}) 
         = u_{n+1} 
         & \mbox{in}\ \Omega, 
 \\[1mm]
         z_{n+1} = \frac{v_{n+1}-v_{n}}{h},\ v_{n+1} = \frac{\varphi_{n+1}-\varphi_{n}}{h} 
         & \mbox{in}\ \Omega, 
 \\[1mm]
         \partial_{\nu}u_{n+1} + u_{n+1} = g_{n+1}                                   
         & \mbox{on}\ \partial\Omega 
     \end{cases}
\end{equation*}
for $n=0, ... , N-1$, where $h=\frac{T}{N}$, $N \in \mathbb{N}$, 
\begin{align*}
&\theta_{j} := -\frac{1}{u_{j}}
\end{align*}
for $j=0, 1, ..., N$, and   
$f_{k} := \frac{1}{h}\int_{(k-1)h}^{kh} f(s)\,ds$, 
$g_{k} := \frac{1}{h}\int_{(k-1)h}^{kh} g(s)\,ds$  
for $k = 1, ... , N$. 
Indeed, we can show existence for \ref{Pn}.   
\begin{thm}\label{maintheorem2}
Assume that {\rm {\bf A1}-{\bf A4}} hold. 
Then there exists $h_{0} \in (0, 1]$  
such that for all $h \in (0, h_{0})$ 
there exists a unique solution of {\rm \ref{Pn}} satisfying 
\[
u_{n+1} \in H^2(\Omega),\ \varphi_{n+1} \in L^{\infty}(\Omega) 
\quad \mbox{for}\ n = 0, ..., N-1. 
\]
\end{thm}
Putting 
\begin{align}
&\widehat{\theta}_{h}(t) := \theta_{n} + \frac{\theta_{n+1} - \theta_{n}}{h}(t-nh), 
\label{hat1} 
\\[2mm]  
&\widehat{\varphi}_{h}(t) := \varphi_{n} + \frac{\varphi_{n+1}-\varphi_{n}}{h}(t-nh), 
\label{hat2} 
\\[2mm]  
&\widehat{v}_{h}(t) := v_{n} + \frac{v_{n+1}-v_{n}}{h}(t-nh)  
\label{hat3} 
\end{align}
for $t \in [nh, (n+1)h]$, $n = 0, ..., N-1$, 
and 
\begin{align}
&\overline{u}_{h}(t) := u_{n+1},\  
\overline{\theta}_{h} (t) := \theta_{n+1},\ 
\overline{\varphi}_{h} (t) := \varphi_{n+1},\ 
\underline{\varphi}_{h} (t) := \varphi_{n},\ \label{line1}   
\\[2mm]
&\overline{v}_{h} (t) := v_{n+1},\ 
\overline{z}_{h} (t) := z_{n+1},\ 
\overline{f}_{h}(t) := f_{n+1}  
\label{line2}   
\end{align}
for \ $t \in (nh, (n+1)h]$, $n=0, ..., N-1$, 
we can rewrite \ref{Pn} as  
\begin{equation*}\tag*{(P)$_{h}$}\label{Ph}
     \begin{cases}
          (\widehat{\theta}_{h})_{t} + (\widehat{\varphi}_{h})_{t} 
          - \Delta\overline{u}_{h} = \overline{f}_h    
         & \mbox{in}\ \Omega\times(0, T), 
 \\[2mm]
         \overline{z}_{h} + \overline{v}_{h} 
         + a(\cdot)\underline{\varphi}_{h} - J\ast\underline{\varphi}_{h} 
         + \beta(\overline{\varphi}_{h}) + \pi(\overline{\varphi}_{h}) 
         = \overline{u}_{h} 
         & \mbox{in}\ \Omega\times(0, T), 
\\[2mm]
         \overline{z}_{h} = (\widehat{v}_{h})_{t},\ \overline{v}_{h} = (\widehat{\varphi}_{h})_{t} 
         & \mbox{in}\ \Omega\times(0, T), 
 \\[2mm]
         \overline{\theta}_{h} = - \frac{1}{\overline{u}_{h}} 
         & \mbox{in}\ \Omega\times(0, T),  
 \\[2mm]
         \partial_{\nu}\overline{u}_{h} + \overline{u}_{h} = \overline{g}_{h}                                    
         & \mbox{on}\ \partial\Omega\times(0, T),
 \\[2mm]
        \widehat{\theta}_{h}(0) = \theta_{0},\ 
        \widehat{\varphi}_{h}(0) = \varphi_0,\    
        \widehat{v}_{h}(0) = v_0                                     
         & \mbox{in}\ \Omega.  
     \end{cases}
 \end{equation*}
Here we can check directly the following identities by \eqref{hat1}-\eqref{line2}:   
\begin{align}
&\|\widehat{\theta}_{h}\|_{L^{\infty}(0, T; L^2(\Omega))} 
= \max\{
   \|\theta_{0}\|_{L^2(\Omega)}, \|\overline{\theta}_{h}\|_{L^{\infty}(0, T; L^2(\Omega))}
   \}, 
\label{tool1} \\[1mm] 
&\|\widehat{\varphi}_{h}\|_{L^{\infty}(0, T; L^{\infty}(\Omega))} 
= \max\{ \|\varphi_{0}\|_{L^{\infty}(\Omega)}, 
                            \|\overline{\varphi}_{h}\|_{L^{\infty}(0, T; L^{\infty}(\Omega))} \},  
\label{tool2}  \\[1mm] 
&\|\widehat{v}_{h}\|_{L^{\infty}(0, T; L^{\infty}(\Omega))} 
= \max\{ \|v_{0}\|_{L^{\infty}(\Omega)}, 
                                     \|\overline{v}_{h}\|_{L^{\infty}(0, T; L^{\infty}(\Omega))} \},  
\label{tool3}  \\[1mm]  
&\|\overline{\theta}_{h} - \widehat{\theta}_{h}\|_{L^2(0, T; {(H^1(\Omega))}^*)}^2 
= \frac{h^2}{3}\|(\widehat{\theta}_{h})_{t}\|_{L^2(0, T; {(H^1(\Omega))}^*)}^2, 
\label{tool4}  \\[1mm] 
&\|\overline{\varphi}_{h} - \widehat{\varphi}_{h}\|_{L^{\infty}(0, T; L^{\infty}(\Omega))} 
= h\|(\widehat{\varphi}_{h})_{t}\|_{L^{\infty}(0, T; L^{\infty}(\Omega))} 
= h\|\overline{v}_{h}\|_{L^{\infty}(0, T; L^{\infty}(\Omega))}, 
\label{tool5} \\[1mm]
&\|\overline{v}_{h} - \widehat{v}_{h}\|_{L^2(0, T; L^2(\Omega))}^2 
= \frac{h^2}{3}\|(\widehat{v}_{h})_{t}\|_{L^2(0, T; L^2(\Omega))}^2 
= \frac{h^2}{3}\|\overline{z}_{h}\|_{L^2(0, T; L^2(\Omega))}^2,     
\label{tool6}  \\[1mm] 
&\underline{\varphi}_{h} = \overline{\varphi}_{h} - h(\widehat{\varphi}_{h})_{t}. 
\label{tool7}
\end{align}

\bigskip

\noindent 
We can prove Theorem \ref{maintheorem1} 
by passing to the limit in \ref{Ph} as $h \searrow 0$. 
Moreover, we can obtain the following theorem which asserts an error estimate 
between solutions of \eqref{P} and solutions of \ref{Ph}.

\medskip

\begin{thm}\label{maintheorem3} 
Let $h_{0}$ be as in Theorem \ref{maintheorem1}. 
Assume that {\rm {\bf A1}-{\bf A4}} hold. 
Assume further that 
$f \in W^{1,1}(0, T; L^2(\Omega))$ 
and $g \in W^{1,1}(0, T; L^2(\partial\Omega))$. 
Then there exist constants $h_{00} \in (0, h_{0})$ and $M>0$
depending on the data such that 
\begin{align*} 
\|1\star(\overline{u}_{h} - u)\|_{C([0, T]; H^1(\Omega))}
+ \|\widehat{\varphi}_{h} - \varphi\|_{C([0, T]; L^2(\Omega))}
  + \|\widehat{v}_{h} - \varphi_{t}\|_{C([0, T]; L^2(\Omega))}
\leq M h^{1/2}    
\end{align*}
for all $h \in (0, h_{00})$, where $(1 \star w)(t) := \int_{0}^{t} w(s)\,ds$ 
for vector-valued functions $w$ summable in $(0, T)$.   
\end{thm}

\vspace{10pt}

\section{Existence for the discrete problem}\label{Sec3}
In this section we will show Theorem \ref{maintheorem2}. 
\begin{lem}\label{elliptic1}
For all 
$h >0$, $G \in L^2(\Omega)$, 
$G_{\partial\Omega} \in H^{1/2}(\partial\Omega)$, 
if $G_{\partial\Omega} \leq 0$ a.e.\ on $\partial\Omega$, 
then there exists a unique function $u \in H^2(\Omega)$ satisfying  
\[
u < 0 \ \ \mbox{a.e.\ in}\ \Omega,\quad 
- \frac{1}{u} - h\Delta u = G\ \ \mbox{a.e.\ in}\ \Omega,\quad 
\partial_{\nu} u + u = G_{\partial\Omega} \ \ \mbox{a.e.\ on}\ \partial\Omega. 
\] 
\end{lem}
\begin{proof}
We set the operator ${\cal A} : D({\cal A}) \subset L^2(\Omega) \to L^2(\Omega)$ as
\[
{\cal A}u := - \Delta u - u  
\quad \mbox{for}\ 
u \in D({\cal A}) := \{ u \in H^2(\Omega) \ |\ \partial_{\nu} u + u = G_{\partial\Omega} 
\ \ \mbox{a.e.\ on}\ \partial\Omega \}. 
\]
Then this operator is maximal monotone. 
Also, we define 
the operator ${\cal B} : D({\cal B}) \subset L^2(\Omega) \to L^2(\Omega)$ as 
\[
{\cal B}u := - \frac{h^{-1}}{u}  
\quad \mbox{for}\ 
u \in D({\cal B}) := \{ u \in L^2(\Omega) \ |\ u < 0 \ \  \mbox{a.e.\ in}\ \Omega \}. 
\]
Then this operator is maximal monotone. 
Now we set the function $b : D(b) \subset \mathbb{R} \to \mathbb{R}$ as 
$b(r) := - \frac{h^{-1}}{r}$ for $r \in D(b) := \{r \in \mathbb{R} \ |\ r < 0 \}$. 
Let $\lambda > 0$, 
let ${\cal B}_{\lambda}$ be the Yosida approximation of ${\cal B}$ 
and let $b_{\lambda}$ be the Yosida approximation of $b$ on $\mathbb{R}$. 
Then, noting that 
$b_{\lambda}$ is monotone, 
$u = \lambda b_{\lambda}(u) + (1 + \lambda b)^{-1}(u)$,  
$b_{\lambda}(u) = -\frac{h^{-1}}{(1 + \lambda b)^{-1}(u)} > 0$,  
and $G_{\partial\Omega} \leq 0$ a.e.\ on $\partial\Omega$, 
we can confirm that 
\begin{align*}
&({\cal A}u, {\cal B}_{\lambda}u)_{L^2(\Omega)} 
\notag \\[2mm] 
&= \int_{\Omega} b_{\lambda}'(u)|\nabla u|^2 
    + \int_{\partial\Omega} ub_{\lambda}(u) 
    - \int_{\partial\Omega} G_{\partial\Omega}b_{\lambda}(u) 
    - \int_{\Omega} ub_{\lambda}(u) 
\notag \\ 
&\geq \lambda\|b_{\lambda}(u)\|_{L^2(\partial\Omega)}^2 
         - h^{-1}|\partial\Omega| 
         - \lambda\|b_{\lambda}(u)\|_{L^2(\Omega)}^2 
         + h^{-1}|\Omega| 
\notag \\ 
&\geq - \max\{1, h^{-1}|\partial\Omega|\}
                                          (\lambda\|{\cal B}_{\lambda}(u)\|_{L^2(\Omega)}^2 + 1)
\end{align*}
for all $u \in D({\cal A})$ and all $\lambda > 0$. 
Therefore we can conclude that 
the operator ${\cal A} + {\cal B}$ is maximal monotone 
(see e.g., Barbu \cite[Theorem 2.7]{Barbu2}).
\end{proof}
\begin{lem}\label{elliptic2}
For all $G \in L^2(\Omega)$ 
and all $h \in (0, \min\{1, 1/\|\pi'\|_{L^{\infty}(\mathbb{R})}\})$
there exists a unique solution $\varphi \in L^2(\Omega)$ of the equation 
\[
\varphi + h\varphi + h^{2} \beta(\varphi) + h^{2} \pi(\varphi) = G 
\quad \mbox{a.e.\ in}\ \Omega.  
\]
\end{lem}
\begin{proof}
We can show this lemma in reference to \cite[Lemma 2.1]{K7}.   
\end{proof}

\begin{prth2.1}
We can rewrite \ref{Pn} as 
\begin{equation*}\tag*{(Q)$_{n}$}\label{Qn}
     \begin{cases}
         -\frac{1}{u_{n+1}} - h\Delta u_{n+1} 
         = - \varphi_{n+1} 
            + hf_{n+1} + \varphi_{n} + \theta_{n},  
      \quad \partial_{\nu} u_{n+1} + u_{n+1} = g_{n+1}, 
     \\[5mm] 
         \varphi_{n+1} + h\varphi_{n+1} + h^2 \beta(\varphi_{n+1}) + h^2 \pi(\varphi_{n+1}) 
         \\[2mm]
         = h^2 u_{n+1} + \varphi_{n} + hv_{n} + h\varphi_{n}  
            -h^2 a(\cdot)\varphi_{n} + h^2 J\ast\varphi_{n}.  
     \end{cases}
 \end{equation*}
To prove Theorem \ref{maintheorem2} 
it suffices to establish existence and uniqueness
of solutions to \ref{Qn} in the case that $n = 0$.  
Let $h \in (0, \min\{1, 1/\|\pi'\|_{L^{\infty}(\mathbb{R})}\})$. 
Then, owing to Lemma \ref{elliptic1},  
for all $\varphi \in L^2(\Omega)$ 
there exists a unique function $\overline{u} \in H^2(\Omega)$ 
such that    
\begin{align}\label{c1}
 -\frac{1}{\overline{u}} - h\Delta \overline{u} 
         = - \varphi 
            + hf_{1} + \varphi_{0} + \theta_{0},  
      \quad \partial_{\nu} \overline{u} + \overline{u} = g_{1}.  
\end{align}
Also, we see from Lemma \ref{elliptic2} that 
for all $u \in L^2(\Omega)$ 
there exists a unique function $\overline{\varphi} \in L^2(\Omega)$ 
such that   
\begin{align}\label{c2}
\overline{\varphi} + h\overline{\varphi} 
+ h^2 \beta(\overline{\varphi}) + h^2 \pi(\overline{\varphi}) 
= h^2 u + \varphi_{0} + hv_{0} + h\varphi_{0}  
            -h^2 a(\cdot)\varphi_{0} + h^2 J\ast\varphi_{0}. 
\end{align}
Thus we can set $\Phi : L^2(\Omega) \to L^2(\Omega)$, 
$\Psi : L^2(\Omega) \to L^2(\Omega)$ 
and $B : L^2(\Omega) \to L^2(\Omega)$ as   
\[
\Phi\varphi = \overline{u},\ \Psi u = \overline{\varphi} 
\quad \mbox{for}\ \varphi, u \in L^2(\Omega)
\]
and 
\[
B = \Psi \circ \Phi, 
\]
respectively. 
Moreover, we can obtain that for all $\varphi, \widetilde{\varphi} \in L^2(\Omega)$, 
\begin{equation*}
\|B\varphi - B\widetilde{\varphi}\|_{L^2(\Omega)} 
\leq \frac{C_{1} h}{1 + h - \|\pi'\|_{L^{\infty}(\mathbb{R})}h^2}
                                                    \|\varphi - \widetilde{\varphi}\|_{L^2(\Omega)}   
\end{equation*}
(cf.\ \cite[Proof of Theorem 1.2]{K8}). 
Here there exists 
$h_{01} \in (0, \min\{1, 1/\|\pi'\|_{L^{\infty}(\mathbb{R})}\})$  
such that 
\[
\frac{C_{1} h}{1 + h - \|\pi'\|_{L^{\infty}(\mathbb{R})}h^2} \in (0, 1) 
\]
for all $h \in (0, h_{01})$. 
Hence $B : L^2(\Omega) \to L^2(\Omega)$ 
is a contraction mapping in $L^2(\Omega)$ 
for all $h \in (0, h_{01})$    
and then 
it follows from the Banach fixed-point theorem that 
for all $h \in (0, h_{01})$  
there exists a unique function $\varphi_{1} \in L^2(\Omega)$ 
such that $\varphi_{1} = B\varphi_{1} \in L^2(\Omega)$. 
Thus, for all $h \in (0, h_{01})$,   
putting $u_{1} := \Phi\varphi_{1} \in H^2(\Omega)$ implies that 
there exists a unique pair $(u_{1}, \varphi_{1}) \in (L^2(\Omega))^2$ 
satisfying \ref{Qn} in the case that $n = 0$. 
Moreover, 
we can prove that 
there exists $h_{0} \in (0, h_{01})$ such that 
for all $h \in (0, h_{0})$ there exists a constant $C_{1} = C_{1}(h) > 0$ 
such that $|\varphi_{1}(x)| \leq C_{1}$ for a.a.\ $x \in \Omega$ 
(cf.\ \cite[Proof of Theorem 1.2]{K8}). 
\qed
\end{prth2.1}

\vspace{10pt}
 
\section{Uniform estimates for the discrete problem}\label{Sec4}

In this section we will derive a priori estimates for \ref{Ph}. 
\begin{lem}\label{esth1}
Let $h_{0}$ be as in Theorem \ref{maintheorem2}.  
Then there exist constants $h_{1} \in (0, h_{0})$ and $C>0$ depending on the data 
such that 
\begin{align*}
&\|\overline{\varphi}_{h}\|_{L^{\infty}(0, T; L^2(\Omega))}^2 
+ \|\overline{v}_{h}\|_{L^{\infty}(0, T; L^2(\Omega))}^2 
+ \|\overline{u}_{h}\|_{L^2(0, T; H^1(\Omega))}^2 
+ \|\overline{\theta}_{h}\|_{L^{\infty}(0, T; L^1(\Omega))}  
\notag \\ 
& + \|\ln\overline{\theta}_{h}\|_{L^{\infty}(0, T; L^1(\Omega))} 
\leq C 
\end{align*}
for all $h \in (0, h_{1})$.   
\end{lem}
\begin{proof}
Multiplying the identity $v_{n+1} = \frac{\varphi_{n+1} - \varphi_{n}}{h}$ 
by $h\varphi_{n+1}$ means that 
\begin{align}\label{e1}
\frac{1}{2}\|\varphi_{n+1}\|_{L^2(\Omega)}^2 
- \frac{1}{2}\|\varphi_{n}\|_{L^2(\Omega)}^2 
+ \frac{1}{2}\|\varphi_{n+1} - \varphi_{n}\|_{L^2(\Omega)}^2 
= h(\varphi_{n+1}, v_{n+1})_{L^2(\Omega)}. 
\end{align}
We test the second equation in \ref{Pn} by $hv_{n+1}$ to infer that 
\begin{align}\label{e2}
&\frac{1}{2}\|v_{n+1}\|_{L^2(\Omega)}^2 
- \frac{1}{2}\|v_{n}\|_{L^2(\Omega)}^2 
+ \frac{1}{2}\|v_{n+1} - v_{n}\|_{L^2(\Omega)}^2 
+ h\|v_{n+1}\|_{L^2(\Omega)}^2 
\notag \\
&+ (\beta(\varphi_{n+1}), \varphi_{n+1} - \varphi_{n})_{L^2(\Omega)} 
\notag \\[2mm] 
&= h(u_{n+1}, v_{n+1})_{L^2(\Omega)} 
     - h(\pi(\varphi_{n+1}), v_{n+1})_{L^2(\Omega)} 
     - h(a(\cdot)\varphi_{n} - J\ast\varphi_{n}, v_{n+1})_{L^2(\Omega)}. 
\end{align}
Here the condition {\bf A2} leads to the inequality 
\begin{align}\label{e3}
(\beta(\varphi_{n+1}), \varphi_{n+1} - \varphi_{n})_{L^2(\Omega)} 
\geq \|\widehat{\beta}(\varphi_{n+1})\|_{L^1(\Omega)} 
        - \|\widehat{\beta}(\varphi_{n})\|_{L^1(\Omega)}.  
\end{align}
Thus we deduce from \eqref{e1}-\eqref{e3}, the Young inequality, 
{\bf A1}, and {\bf A3} that 
there exists a constant $C_{1} > 0$ such that 
\begin{align}\label{e4}
&\frac{1}{2}\|\varphi_{n+1}\|_{L^2(\Omega)}^2 
- \frac{1}{2}\|\varphi_{n}\|_{L^2(\Omega)}^2 
+ \frac{1}{2}\|\varphi_{n+1} - \varphi_{n}\|_{L^2(\Omega)}^2 
\notag \\ 
&+ \frac{1}{2}\|v_{n+1}\|_{L^2(\Omega)}^2 
- \frac{1}{2}\|v_{n}\|_{L^2(\Omega)}^2 
+ \frac{1}{2}\|v_{n+1} - v_{n}\|_{L^2(\Omega)}^2 
+ h\|v_{n+1}\|_{L^2(\Omega)}^2 
\notag \\ 
&+ \|\widehat{\beta}(\varphi_{n+1})\|_{L^1(\Omega)} 
        - \|\widehat{\beta}(\varphi_{n})\|_{L^1(\Omega)} 
\notag \\[2mm]
&\leq h(u_{n+1}, v_{n+1})_{L^2(\Omega)} 
         + C_{1}h + C_{1}\|\varphi_{n+1}\|_{L^2(\Omega)}^2 
         + C_{1}\|\varphi_{n}\|_{L^2(\Omega)}^2 
         + C_{1}\|v_{n+1}\|_{L^2(\Omega)}^2 
\end{align}
for all $h \in (0, h_{0})$. 
Next we multiply the first equation in \ref{Pn} by $h(1 + u_{n+1})$ to obtain that 
\begin{align}\label{e5}
&(\theta_{n+1} - \theta_{n}, 1 + u_{n+1})_{L^2(\Omega)} 
+ h\int_{\Omega}|\nabla u_{n+1}|^2 
+ h\int_{\partial\Omega} |u_{n+1}|^2  
\notag \\[5mm] 
&= h(f_{n+1}, 1 + u_{n+1})_{L^2(\Omega)} 
     - h(u_{n+1}, v_{n+1})_{L^2(\Omega)} 
\notag \\ 
  &\,\quad - h(v_{n+1}, 1)_{L^2(\Omega)} 
     - h\int_{\partial\Omega} u_{n+1} + h\int_{\partial\Omega} g_{n+1}(1 + u_{n+1}). 
\end{align}
Here, noting that $u_{n+1} = - \frac{1}{\theta_{n+1}}$ 
and $r - 1 \geq \ln r$ for all $r > 0$, 
we have that 
\begin{align}\label{e6}
&(\theta_{n+1} - \theta_{n}, 1 + u_{n+1})_{L^2(\Omega)} 
\notag \\[3mm]
&= \|\theta_{n+1}\|_{L^1(\Omega)} - \|\theta_{n}\|_{L^1(\Omega)} 
     + (\theta_{n+1} - \theta_{n}, u_{n+1})_{L^2(\Omega)} 
\notag \\ 
&= \|\theta_{n+1}\|_{L^1(\Omega)} - \|\theta_{n}\|_{L^1(\Omega)} 
     + \int_{\Omega}\left(\frac{\theta_{n}}{\theta_{n+1}} - 1 \right) 
\notag \\ 
&\geq \|\theta_{n+1}\|_{L^1(\Omega)} - \|\theta_{n}\|_{L^1(\Omega)} 
     + \int_{\Omega} \ln\frac{\theta_{n}}{\theta_{n+1}}  
\notag \\ 
&= \|\theta_{n+1}\|_{L^1(\Omega)} - \|\theta_{n}\|_{L^1(\Omega)} 
     + \int_{\Omega} (- \ln\theta_{n+1} + \ln\theta_{n}). 
\end{align}
There exist constants $C_{*}, C^{*} > 0$ such that 
\begin{align}\label{e7}
C_{*}(\|\nabla w\|_{L^2(\Omega)}^2 + \|w\|_{L^2(\partial\Omega)}^2) 
\leq \|w\|_{H^1(\Omega)}^2 
\leq C^{*}(\|\nabla w\|_{L^2(\Omega)}^2 + \|w\|_{L^2(\partial\Omega)}^2) 
\end{align}
for all $w \in H^1(\Omega)$. 
Therefore we see from \eqref{e5}-\eqref{e7} 
and the Young inequality that 
there exists a constant $C_{2} > 0$ such that  
\begin{align}\label{e8}
&\|\theta_{n+1}\|_{L^1(\Omega)} - \|\theta_{n}\|_{L^1(\Omega)} 
     + \int_{\Omega} (- \ln\theta_{n+1} + \ln\theta_{n})
+ \frac{1}{2C^{*}}h\|u_{n+1}\|_{H^1(\Omega)}^2 
\notag \\[2mm] 
& \leq - h(u_{n+1}, v_{n+1})_{L^2(\Omega)} 
         + C_{2}h 
         + C_{2}h\|f_{n+1}\|_{L^2(\Omega)}^2 
         + C_{2}h\|g_{n+1}\|_{L^2(\partial\Omega)}^2 
\notag \\
   &\,\quad+ C_{2}h\|v_{n+1}\|_{L^2(\Omega)}^2
\end{align}
for all $h \in (0, h_{0})$. 
Therefore we add \eqref{e4} to \eqref{e8} 
and sum over $n = 0, ..., m-1$ with $1 \leq m \leq N$ 
to derive that 
\begin{align}\label{e9}
&\frac{1}{2}\|\varphi_{m}\|_{L^2(\Omega)}^2 
+ \frac{1}{2}\|v_{m}\|_{L^2(\Omega)}^2 
+ \|\widehat{\beta}(\varphi_{m})\|_{L^1(\Omega)} 
\notag \\ 
&+ \|\theta_{m}\|_{L^1(\Omega)} 
- \int_{\Omega} \ln\theta_{m} 
+ \frac{1}{2C^*}h\sum_{n=0}^{m-1}\|u_{n+1}\|_{H^1(\Omega)}^2 
\notag \\[4mm]
&\leq \frac{1}{2}\|\varphi_{0}\|_{L^2(\Omega)}^2 
+ \frac{1}{2}\|v_{0}\|_{L^2(\Omega)}^2 
+ \|\widehat{\beta}(\varphi_{0})\|_{L^1(\Omega)} 
+ \|\theta_{0}\|_{L^1(\Omega)} 
- \int_{\Omega} \ln\theta_{0} 
\notag \\ 
&\,\quad + (C_{1} + C_{2})T 
+ C_{2}h\sum_{n=0}^{m-1}\|f_{n+1}\|_{L^2(\Omega)}^2 
+ C_{2}h\sum_{n=0}^{m-1}\|g_{n+1}\|_{L^2(\partial\Omega)}^2 
\notag \\ 
&\,\quad + 2C_{1}h\sum_{n=0}^{m-1}\|\varphi_{n+1}\|_{L^2(\Omega)}^2 
+ (C_{1} + C_{2})h\sum_{n=0}^{m-1}\|v_{n+1}\|_{L^2(\Omega)}^2.  
\end{align}
On the other hand, it holds that 
\begin{align}\label{e10}
\|\theta_{m}\|_{L^1(\Omega)} 
- \int_{\Omega} \ln\theta_{m} 
= \int_{\Omega}(\theta_{m} - \ln\theta_{m}) 
\geq \frac{1}{3}\int_{\Omega}(\theta_{m} + |\ln\theta_{m}|).  
\end{align}
Thus it follows from \eqref{e9} and \eqref{e10} that 
\begin{align*}
&\left(\frac{1}{2} - 2C_{1}h\right)\|\varphi_{m}\|_{L^2(\Omega)}^2 
+ \left(\frac{1}{2} - (C_{1} + C_{2})h\right)\|v_{m}\|_{L^2(\Omega)}^2 
+ \|\widehat{\beta}(\varphi_{m})\|_{L^1(\Omega)} 
\notag \\ 
&+ \frac{1}{3}\|\theta_{m}\|_{L^1(\Omega)} 
+ \frac{1}{3}\|\ln\theta_{m}\|_{L^1(\Omega)} 
+ \frac{1}{2C^*}h\sum_{n=0}^{m-1}\|u_{n+1}\|_{H^1(\Omega)}^2 
\notag \\[4mm]
&\leq \frac{1}{2}\|\varphi_{0}\|_{L^2(\Omega)}^2 
+ \frac{1}{2}\|v_{0}\|_{L^2(\Omega)}^2 
+ \|\widehat{\beta}(\varphi_{0})\|_{L^1(\Omega)} 
+ \|\theta_{0}\|_{L^1(\Omega)} 
+ \|\ln\theta_{0}\|_{L^1(\Omega)} 
\notag \\ 
&\,\quad + (C_{1} + C_{2})T 
+ C_{2}h\sum_{n=0}^{m-1}\|f_{n+1}\|_{L^2(\Omega)}^2 
+ C_{2}h\sum_{n=0}^{m-1}\|g_{n+1}\|_{L^2(\partial\Omega)}^2 
\notag \\ 
&\,\quad + 2C_{1}h\sum_{j=0}^{m-1}\|\varphi_{j}\|_{L^2(\Omega)}^2 
+ (C_{1} + C_{2})h\sum_{j=0}^{m-1}\|v_{j}\|_{L^2(\Omega)}^2 
\end{align*}
and then there exist constants $C_{3} > 0$ and $h_{1} \in (0, h_{0})$ 
such that 
\begin{align*}
&\|\varphi_{m}\|_{L^2(\Omega)}^2 
+ \|v_{m}\|_{L^2(\Omega)}^2 
+ \|\widehat{\beta}(\varphi_{m})\|_{L^1(\Omega)} 
\notag \\ 
&+ \|\theta_{m}\|_{L^1(\Omega)} 
+ \|\ln\theta_{m}\|_{L^1(\Omega)} 
+ h\sum_{n=0}^{m-1}\|u_{n+1}\|_{H^1(\Omega)}^2 
\notag \\[2mm]
&\leq C_{3} + C_{3}h\sum_{j=0}^{m-1}\|\varphi_{j}\|_{L^2(\Omega)}^2 
+ C_{3}h\sum_{j=0}^{m-1}\|v_{j}\|_{L^2(\Omega)}^2 
\end{align*}
for all $h \in (0, h_{1})$ 
and $m = 1, ..., N$. 
Therefore, owing to the discrete Gronwall lemma (see e.g., \cite[Prop.\ 2.2.1]{Jerome}), 
there exists a constant $C_{4} > 0$ such that 
\begin{align*}
&\|\varphi_{m}\|_{L^2(\Omega)}^2 
+ \|v_{m}\|_{L^2(\Omega)}^2 
+ \|\widehat{\beta}(\varphi_{m})\|_{L^1(\Omega)} 
\notag \\ 
&+ \|\theta_{m}\|_{L^1(\Omega)} 
+ \|\ln\theta_{m}\|_{L^1(\Omega)} 
+ h\sum_{n=0}^{m-1}\|u_{n+1}\|_{H^1(\Omega)}^2 
\leq C_{4} 
\end{align*}
for all $h \in (0, h_{1})$ and $m = 1, ..., N$.  
\end{proof}

\begin{lem}\label{esth2}
Let $h_{1}$ be as in Lemma \ref{esth1}. 
Then there exist constants $h_{2} \in (0, h_{1})$ and $C>0$ depending on the data 
such that
\begin{align*}
\|\overline{\theta}_{h}\|_{L^{\infty}(0, T; L^2(\Omega))}^2 
+ \|\ln\overline{\theta}_{h}\|_{L^{\infty}(0, T; H^1(\Omega))}^2 
\leq C 
\end{align*}
for all $h \in (0, h_{2})$.  
\end{lem}
\begin{proof}
Testing the first equation in \ref{Pn} by $h\theta_{n+1}$ 
leads to the identity 
\begin{align}\label{f1}
&\frac{1}{2}\|\theta_{n+1}\|_{L^2(\Omega)}^2 
- \frac{1}{2}\|\theta_{n}\|_{L^2(\Omega)}^2 
+ \frac{1}{2}\|\theta_{n+1} - \theta_{n}\|_{L^2(\Omega)}^2 
+ h(-\Delta u_{n+1}, \theta_{n+1})_{L^2(\Omega)} 
\notag \\ 
&= h(f_{n+1}, \theta_{n+1})_{L^2(\Omega)} - h(v_{n+1}, \theta_{n+1})_{L^2(\Omega)}. 
\end{align}
Here, since $u_{n+1} = - \frac{1}{\theta_{n+1}}$, $\theta_{n+1} > 0$ and $g_{n+1} \leq 0$,  
we have that 
\begin{align}\label{f2}
&h(-\Delta u_{n+1}, \theta_{n+1})_{L^2(\Omega)} 
\notag \\[3mm]
&= h\int_{\Omega} \nabla u_{n+1} \cdot \nabla \theta_{n+1} 
     + h\int_{\partial\Omega} u_{n+1}\theta_{n+1} 
     - h\int_{\partial\Omega} g_{n+1}\theta_{n+1} 
\notag \\ 
&\geq h\int_{\Omega} |\nabla\ln\theta_{n+1}|^2 - h|\partial\Omega|.  
\end{align}
Therefore we can verify that Lemma \ref{esth2} holds 
by combining \eqref{f1}, \eqref{f2}, 
by summing over $n = 0, ..., m-1$ with $1 \leq m \leq N$, 
by applying the discrete Gronwall lemma, 
by Lemma \ref{esth1} and the Poincar\'e--Wirtinger inequality.  
\end{proof}

\begin{lem}\label{esth3}
Let $h_{2}$ be as in Lemma \ref{esth2}. 
Then there exists a constant  $C>0$ depending on the data such that
\begin{align*}
\|(\widehat{\theta}_{h})_{t}\|_{L^2(0, T; {(H^1(\Omega))}^*)}
\leq C 
\end{align*}
for all $h \in (0, h_{2})$. 
\end{lem}
\begin{proof}
This lemma can be obtained 
by the first equation in \ref{Ph} and Lemma \ref{esth1}.  
\end{proof}

\begin{lem}\label{esth4}
Let $h_{2}$ be as in Lemma \ref{esth2}. 
Then there exists a constant  $C>0$ depending on the data such that
\begin{align*}
h\max_{1 \leq m \leq N}
\left\|
\sum_{n = 0}^{m-1}(-u_{n+1})
\right\|_{H^2(\Omega)}
\leq C 
\end{align*}
for all $h \in (0, h_{2})$.
\end{lem}
\begin{proof}
We can prove this lemma 
by Lemmas \ref{esth1}, \ref{esth2} 
and the elliptic regularity theory (cf.\ \cite[Proof of Lemma 4.5]{K8}). 
\end{proof}

\begin{lem}\label{esth5}
Let $h_{2}$ be as in Lemma \ref{esth2}. 
Then there exist constants $h_{3} \in (0, h_{2})$ and $C>0$ depending on the data 
such that
\begin{align*}
\|\overline{\varphi}_{h}\|_{L^{\infty}(\Omega\times(0, T))}^2 
+ \|\overline{v}_{h}\|_{L^{\infty}(\Omega\times(0, T))}^2 
\leq C 
\end{align*}
for all $h \in (0, h_{3})$. 
\end{lem}
\begin{proof}
In reference to \cite[Proof of Lemma 4.6]{K8}, 
we can confirm that 
there exists a constant $C_{1} > 0$ such that 
\begin{align}\label{g1}
&\frac{1}{2}|\varphi_{m}(x)|^2 + \frac{1}{2}|v_{m}(x)|^2 
\notag \\[4mm] 
&\leq h\sum_{n=0}^{m-1} u_{n+1}(x)v_{n+1}(x) 
\notag \\ 
   &\,\quad + C_{1}h\sum_{n=0}^{m-1} \|\varphi_{n+1}\|_{L^{\infty}(\Omega)}^2 
         + C_{1}h\sum_{n=0}^{m-1} \|v_{n+1}\|_{L^{\infty}(\Omega)}^2 
         + C_{1} 
\end{align}
for all $h \in (0, h_{2})$ and for a.a.\ $x \in \Omega$, $m = 1, ..., N$. 
Here, noting that $- u_{j} > 0$ a.e.\ in $\Omega$ 
for $j = 0, 1, ..., N$, 
we deduce from Lemma \ref{esth4} 
and the continuity of the embedding $H^2(\Omega) \hookrightarrow L^{\infty}(\Omega)$ 
that there exists a constant $C_{2} >0$ such that 
\begin{align}\label{g2}
h\sum_{n=0}^{m-1} u_{n+1}(x)v_{n+1}(x)  
&= h\sum_{n=0}^{m-1}(-u_{n+1}(x))(-v_{n+1}(x)) 
\notag \\ 
&\leq \Bigl(\max_{1 \leq m \leq N}\|-v_{m}\|_{L^{\infty}(\Omega)}\Bigr) 
       h\sum_{n=0}^{m-1}(-u_{n+1}(x))
\notag \\
&\leq \Bigl(\max_{1 \leq m \leq N}\|v_{m}\|_{L^{\infty}(\Omega)}\Bigr) 
       h\left\|\sum_{n=0}^{m-1}(-u_{n+1})\right\|_{L^{\infty}(\Omega)} 
\notag \\ 
&\leq C_{2}\max_{1 \leq m \leq N}\|v_{m}\|_{L^{\infty}(\Omega)}
\end{align}
for all $h \in (0, h_{2})$ and for a.a.\ $x \in \Omega$, $m = 1, ..., N$. 
Thus we see from \eqref{g1} and \eqref{g2} that 
\begin{align*}
&\frac{1}{2}|\varphi_{m}(x)|^2 + \frac{1}{2}|v_{m}(x)|^2 
\notag \\
&\leq C_{2}\max_{1 \leq m \leq N}\|v_{m}\|_{L^{\infty}(\Omega)} 
+ C_{1}h\sum_{n=0}^{m-1} \|\varphi_{n+1}\|_{L^{\infty}(\Omega)}^2 
         + C_{1}h\sum_{n=0}^{m-1} \|v_{n+1}\|_{L^{\infty}(\Omega)}^2 
         + C_{1} 
\end{align*}
for a.a.\ $x \in \Omega$ and for all $h \in (0, h_{2})$, $m = 1, ..., N$,  
whence the inequality 
\begin{align*}
&\frac{1}{2}\|\varphi_{m}\|_{L^{\infty}(\Omega)}^2 
+ \frac{1}{2}\|v_{m}\|_{L^{\infty}(\Omega)}^2 
\notag \\ 
&\leq C_{2}\max_{1 \leq m \leq N}\|v_{m}\|_{L^{\infty}(\Omega)} 
+ C_{1}h\sum_{n=0}^{m-1} \|\varphi_{n+1}\|_{L^{\infty}(\Omega)}^2 
         + C_{1}h\sum_{n=0}^{m-1} \|v_{n+1}\|_{L^{\infty}(\Omega)}^2 
         + C_{1} 
\end{align*}
holds. 
Then there exist constants $h_{3} \in (0, h_{2})$ and $C_{3} > 0$ 
such that  
\begin{align*}
&\|\varphi_{m}\|_{L^{\infty}(\Omega)}^2 + \|v_{m}\|_{L^{\infty}(\Omega)}^2 
\notag \\  
&\leq C_{3}\max_{1 \leq m \leq N}\|v_{m}\|_{L^{\infty}(\Omega)} 
+ C_{3}h\sum_{j=0}^{m-1} \|\varphi_{j}\|_{L^{\infty}(\Omega)}^2 
         + C_{3}h\sum_{j=0}^{m-1} \|v_{j}\|_{L^{\infty}(\Omega)}^2 
         + C_{3} 
\end{align*}
for all $h \in (0, h_{3})$ and $m = 1, ..., N$. 
Hence by the discrete Gronwall lemma 
there exists a constant $C_{4} > 0$ such that 
\begin{align*}
\|\varphi_{m}\|_{L^{\infty}(\Omega)}^2 + \|v_{m}\|_{L^{\infty}(\Omega)}^2 
\leq C_{4} + C_{4}\max_{1 \leq m \leq N}\|v_{m}\|_{L^{\infty}(\Omega)} 
\end{align*}
for all $h \in (0, h_{3})$ and $m = 1, ..., N$. 
Therefore it holds that 
\begin{align*}
\max_{1 \leq m \leq N}\|\varphi_{m}\|_{L^{\infty}(\Omega)}^2 
+ \max_{1 \leq m \leq N}\|v_{m}\|_{L^{\infty}(\Omega)}^2 
&\leq C_{4} + C_{4}\max_{1 \leq m \leq N}\|v_{m}\|_{L^{\infty}(\Omega)} 
\notag \\ 
&\leq  C_{4} + \frac{1}{2}\max_{1 \leq m \leq N}\|v_{m}\|_{L^{\infty}(\Omega)}^2 
                 + \frac{C_{4}^2}{2},
\end{align*}
which leads to Lemma \ref{esth5}. 
\end{proof}

\begin{lem}\label{esth6}
Let $h_{3}$ be as in Lemma \ref{esth5}. 
Then there exists a constant  $C>0$ depending on the data such that 
\begin{align*}
\|\underline{\varphi}_{h}\|_{L^{\infty}(\Omega\times(0, T))}
\leq C 
\end{align*}
for all $h \in (0, h_{3})$. 
\end{lem}
\begin{proof}
This lemma can be obtained by {\bf A4} and Lemma \ref{esth5}. 
\end{proof}

\begin{lem}\label{esth7}
Let $h_{3}$ be as in Lemma \ref{esth5}. 
Then there exists a constant  $C>0$ depending on the data such that 
\begin{align*}
\|\beta(\overline{\varphi}_{h})\|_{L^{\infty}(\Omega\times(0, T))}
\leq C 
\end{align*}
for all $h \in (0, h_{3})$. 
\end{lem}
\begin{proof}
We can prove this lemma by the continuity of $\beta$ and Lemma \ref{esth5}. 
\end{proof}

\begin{lem}\label{esth8}
Let $h_{3}$ be as in Lemma \ref{esth5}. 
Then there exists a constant  $C>0$ depending on the data such that 
\begin{align*}
\|\overline{z}_{h}\|_{L^2(0, T; L^2(\Omega))}
\leq C 
\end{align*}
for all $h \in (0, h_{3})$. 
\end{lem}
\begin{proof}
We can verify that this lemma holds 
by the second equation in \ref{Ph}, 
Lemmas \ref{esth1}, \ref{esth5}-\ref{esth7}, 
the conditions {\bf A1}, {\bf A3}. 
\end{proof}

\begin{lem}\label{esth9}
Let $h_{3}$ be as in Lemma \ref{esth5}. 
Then there exists a constant  $C>0$ depending on the data such that 
\begin{align*}
&
\|\widehat{\theta}_{h}\|_{H^1(0, T; {(H^1(\Omega))}^*) \cap L^{\infty}(0, T; L^2(\Omega))} 
+ \|\widehat{v}_{h}\|_{H^1(0, T; L^2(\Omega)) \cap L^{\infty}(\Omega\times(0, T))} 
+ \|\widehat{\varphi}_{h}\|_{W^{1, \infty}(0, T; L^{\infty}(\Omega))} 
\notag \\
&
\leq C 
\end{align*}
for all $h \in (0, h_{3})$. 
\end{lem}
\begin{proof}
Lemmas \ref{esth2}, \ref{esth3}, \ref{esth5}, \ref{esth8}, 
along with \eqref{tool1}-\eqref{tool3}, 
lead to Lemma \ref{esth9}. 
\end{proof}

\vspace{10pt}

\section{Existence for \eqref{P} and error estimate}\label{Sec5}

In this section we will derive existence and uniqueness of solutions to \eqref{P} 
by passing to the limit in \ref{Ph} as $h \searrow 0$  
and will establish an error estimate 
between solutions of \eqref{P} and solutions of \ref{Ph}. 
\begin{lem}\label{LemC1}
Let $h_{3}$ be as in Lemma \ref{esth5}. 
Then there exists a constant  $M_{1} > 0$ depending on the data such that 
\begin{align}\label{C1}
&\|(1\star(\overline{u}_{h} - \overline{u}_{\tau}))(t)\|_{H^1(\Omega)}^2 
\notag \\[3mm]
&\leq M_{1}(h + \tau) 
         + M_{1}\int_{0}^{t} 
                 \|(1\star(\overline{u}_{h} - \overline{u}_{\tau}))(s)\|_{H^1(\Omega)}^2\,ds 
         + M_{1}\int_{0}^{t} \|\widehat{v}_{h}(s) - \widehat{v}_{\tau}(s)\|_{L^2(\Omega)}^2\,ds
\notag \\
&\,\quad + M_{1}\|\overline{f}_{h} - \overline{f}_{\tau}\|_{L^2(0, T; L^2(\Omega))}^2 
           + M_{1}\|\overline{g}_{h} - \overline{g}_{\tau}\|_{L^2(0, T; L^2(\partial\Omega))}^2 
\end{align}
for all $h, \tau \in (0, h_{3})$ and all $t \in [0, T]$. 
\end{lem}
\begin{proof}
We have that 
\begin{align}\label{h1}
&(\widehat{\theta}_{h} - \widehat{\theta}_{\tau}, w)_{L^2(\Omega)} 
+ (\widehat{\varphi}_{h} - \widehat{\varphi}_{\tau}, w)_{L^2(\Omega)} 
+ \int_{\Omega} \nabla (1\star(\overline{u}_{h} - \overline{u}_{\tau})) \cdot \nabla w 
\notag \\ 
& + \int_{\partial\Omega} (1\star(\overline{u}_{h} - \overline{u}_{\tau}))w 
= ((1\star(\overline{f}_{h} - \overline{f}_{\tau})), w)_{L^2(\Omega)} 
   + \int_{\partial\Omega} (1\star(\overline{g}_{h} - \overline{g}_{\tau}))w 
\notag \\[6mm]  
&\hspace{75mm} \mbox{a.e.\ in}\ (0, T) \ \  \mbox{for all}\ w \in H^1(\Omega).
\end{align}
We take $w = \overline{u}_{h} - \overline{u}_{\tau}$ in \eqref{h1} 
and integrate over $(0, t)$, where $t \in [0, T]$, to infer that 
\begin{align}\label{h2}
&\int_{0}^{t} (\widehat{\theta}_{h}(s) - \widehat{\theta}_{\tau}(s), 
                                    \overline{u}_{h}(s) - \overline{u}_{\tau}(s))_{L^2(\Omega)}\,ds 
+ \int_{0}^{t} (\widehat{\varphi}_{h}(s) - \widehat{\varphi}_{\tau}(s), 
                                    \overline{u}_{h}(s) - \overline{u}_{\tau}(s))_{L^2(\Omega)}\,ds 
\notag \\ 
&+ \frac{1}{2}\|\nabla(1\star(\overline{u}_{h} - \overline{u}_{\tau}))(s)\|_{L^2(\Omega)}^2 
+ \frac{1}{2}\|(1\star(\overline{u}_{h} - \overline{u}_{\tau}))(s)\|_{L^2(\partial\Omega)}^2 
\notag \\[5mm]  
&= 
\int_{0}^{t} ((1\star(\overline{f}_{h} - \overline{f}_{\tau}))(s),  
                         (1\star(\overline{u}_{h} - \overline{u}_{\tau}))(s))_{L^2(\Omega)}\,ds 
\notag \\[2mm]
&\,\quad + \int_{0}^{t}\left(\int_{\partial\Omega} 
                  (1\star(\overline{g}_{h} - \overline{g}_{\tau}))(s)
                   (1\star(\overline{u}_{h} - \overline{u}_{\tau}))(s)
                  \right)\,ds. 
\end{align}
Here we see from the identity $\overline{\theta}_{h} = - \frac{1}{\overline{u}_{h}}$ 
that 
\begin{align}\label{h3}
&\int_{0}^{t} (\widehat{\theta}_{h}(s) - \widehat{\theta}_{\tau}(s), 
                                    \overline{u}_{h}(s) - \overline{u}_{\tau}(s))_{L^2(\Omega)}\,ds 
\notag \\[4mm] 
&= \int_{0}^{t} \left\langle 
                      \widehat{\theta}_{h}(s) - \overline{\theta}_{h}(s), 
                                       \overline{u}_{h}(s) - \overline{u}_{\tau}(s) 
                   \right\rangle_{{(H^1(\Omega))}^*, H^1(\Omega)}\,ds 
\notag \\ 
&\,\quad + \int_{0}^{t} \left\langle 
                                 \overline{\theta}_{\tau}(s) - \widehat{\theta}_{\tau}(s), 
                                                             \overline{u}_{h}(s) - \overline{u}_{\tau}(s)
                              \right\rangle_{{(H^1(\Omega))}^*, H^1(\Omega)}\,ds 
\notag \\ 
&\,\quad + \int_{0}^{t}(\alpha(\overline{u}_{h}(s)) - \alpha(\overline{u}_{\tau}(s)), 
                                   \overline{u}_{h}(s) - \overline{u}_{\tau}(s))_{L^2(\Omega)}\,ds 
\notag \\[3mm] 
&\geq \int_{0}^{t} \left\langle 
                      \widehat{\theta}_{h}(s) - \overline{\theta}_{h}(s), 
                                       \overline{u}_{h}(s) - \overline{u}_{\tau}(s) 
                   \right\rangle_{{(H^1(\Omega))}^*, H^1(\Omega)}\,ds 
\notag \\ 
&\,\quad + \int_{0}^{t} \left\langle 
                                 \overline{\theta}_{\tau}(s) - \widehat{\theta}_{\tau}(s), 
                                                             \overline{u}_{h}(s) - \overline{u}_{\tau}(s)
                              \right\rangle_{{(H^1(\Omega))}^*, H^1(\Omega)}\,ds,  
\end{align}
where $\alpha(r) := - \frac{1}{r}$ 
for $r \in D(\alpha) := \{r \in \mathbb{R} \ |\ r < 0 \}$ 
and the monotonicity of $\alpha$ was used.  
Integrating by parts with respect to time yields that 
\begin{align}\label{h4}
&\int_{0}^{t} (\widehat{\varphi}_{h}(s) - \widehat{\varphi}_{\tau}(s), 
                                    \overline{u}_{h}(s) - \overline{u}_{\tau}(s))_{L^2(\Omega)}\,ds 
\notag \\[3mm]
&= \int_{0}^{t} ((1\star(\overline{v}_{h} - \overline{v}_{\tau}))(s), 
                           (1\star(\overline{u}_{h} - \overline{u}_{\tau}))'(s))_{L^2(\Omega)}\,ds 
\notag \\[4mm] 
&= ((1\star(\overline{v}_{h} - \overline{v}_{\tau}))(t), 
                                  (1\star(\overline{u}_{h} - \overline{u}_{\tau}))(t))_{L^2(\Omega)} 
\notag \\
&\,\quad - \int_{0}^{t} (\overline{v}_{h}(s) - \overline{v}_{\tau}(s), 
                         (1\star(\overline{u}_{h} - \overline{u}_{\tau}))(s))_{L^2(\Omega)}\,ds.  
\end{align}
Also, it holds that 
\begin{align}\label{h5}
&\int_{0}^{t} ((1\star(\overline{f}_{h} - \overline{f}_{\tau}))(s), 
                                    \overline{u}_{h}(s) - \overline{u}_{\tau}(s))_{L^2(\Omega)}\,ds 
\notag \\[4mm] 
&= \int_{0}^{t} ((1\star(\overline{f}_{h} - \overline{f}_{\tau}))(s),  
                            (1\star(\overline{u}_{h} - \overline{u}_{\tau}))'(s))_{L^2(\Omega)}\,ds 
\notag \\[4mm] 
&= ((1\star(\overline{f}_{h} - \overline{f}_{\tau}))(t), 
                                  (1\star(\overline{u}_{h} - \overline{u}_{\tau}))(t))_{L^2(\Omega)} 
\notag \\
&\,\quad - \int_{0}^{t} (\overline{f}_{h}(s) - \overline{f}_{\tau}(s), 
                         (1\star(\overline{u}_{h} - \overline{u}_{\tau}))(s))_{L^2(\Omega)}\,ds 
\end{align}
and 
\begin{align}\label{h6}
&\int_{0}^{t} \left(\int_{\partial\Omega}
                  (1\star(\overline{g}_{h} - \overline{g}_{\tau}))(s)
                                    (\overline{u}_{h}(s) - \overline{u}_{\tau}(s))\right)\,ds 
\notag \\[5mm] 
&= \int_{0}^{t} \left(\int_{\partial\Omega}
                  (1\star(\overline{g}_{h} - \overline{g}_{\tau}))(s)
                         (1\star(\overline{u}_{h} - \overline{u}_{\tau}))'(s)\right)\,ds 
\notag \\[4mm] 
&= \int_{\partial\Omega}
                  (1\star(\overline{g}_{h} - \overline{g}_{\tau}))(t)
                         (1\star(\overline{u}_{h} - \overline{u}_{\tau}))(t)
\notag \\
&\,\quad - \int_{0}^{t}\left(\int_{\partial\Omega}
                  (\overline{g}_{h} - \overline{g}_{\tau})(s)
                         (1\star(\overline{u}_{h} - \overline{u}_{\tau}))(s)\right)\,ds.  
\end{align}
Therefore, 
since 
$\overline{v}_{h} - \overline{v}_{\tau} 
= \overline{v}_{h} - \widehat{v}_{h} + \widehat{v}_{\tau} - \overline{v}_{\tau}
   + \widehat{v}_{h} - \widehat{v}_{\tau}$, 
we can prove Lemma \ref{LemC1} 
by \eqref{h2}-\eqref{h6}, the Schwarz inequality, 
the Young inequality, 
\eqref{tool4}, \eqref{tool6}, 
Lemmas \ref{esth1}, \ref{esth3}, \ref{esth8}. 
\end{proof}

\begin{lem}\label{LemC2}
Let $h_{3}$ be as in Lemma \ref{esth5}. 
Then there exists a constant  $M_{2} > 0$ depending on the data such that 
\begin{align}\label{C2}
&\|\widehat{\varphi}_{h}(t) - \widehat{\varphi}_{\tau}(t)\|_{L^2(\Omega)}^2 
  + \|\widehat{v}_{h}(t) - \widehat{v}_{\tau}(t)\|_{L^2(\Omega)}^2 
\notag \\[3mm]
&\leq M_{2}(h + \tau) 
        + M_{2}\int_{0}^{t} 
              \|\widehat{\varphi}_{h}(s) - \widehat{\varphi}_{\tau}(s)\|_{L^2(\Omega)}^2\,ds
        + M_{2}\int_{0}^{t} \|\widehat{v}_{h}(s) - \widehat{v}_{\tau}(s)\|_{L^2(\Omega)}^2\,ds 
\notag \\
&\,\quad + M_{2}\|(1\star(\overline{u}_{h} - \overline{u}_{\tau}))(t)\|_{H^1(\Omega)}^2 
\end{align}
for all $h, \tau \in (0, h_{3})$ and all $t \in [0, T]$.  
\end{lem}
\begin{proof}
We see from \eqref{tool7} and Lemma \ref{esth1} that 
there exists a constant $C_{1} > 0$ such that 
\begin{align}\label{i1}
&\int_{0}^{t}
    \|\underline{\varphi}_{h}(s) - \underline{\varphi}_{\tau}(s)\|_{L^2(\Omega)}^2\,ds  
\notag \\[3mm] 
& \leq 3\int_{0}^{t}
                \|\overline{\varphi}_{h}(s) - \overline{\varphi}_{\tau}(s)\|_{L^2(\Omega)}^2\,ds   
          + 3h^2\int_{0}^{t}\|(\widehat{\varphi}_{h})_{s}(s) \|_{L^2(\Omega)}^2\,ds  
\notag \\ 
  &\,\quad + 3\tau^2\int_{0}^{t}\|(\widehat{\varphi}_{\tau})_{s}(s) \|_{L^2(\Omega)}^2\,ds 
\notag \\[2mm] 
&\leq 3\int_{0}^{t}
             \|\overline{\varphi}_{h}(s) - \overline{\varphi}_{\tau}(s) \|_{L^2(\Omega)}^2\,ds   
        + C_{1}h^2 + C_{1}\tau^2 
\end{align}
for all $h, \tau \in (0, h_{3})$ and all $t \in [0, T]$. 
Here, owing to \eqref{tool5} and Lemma \ref{esth5}, it holds that 
there exists a constant $C_{2} > 0$ such that  
\begin{align}\label{i2}
&3\int_{0}^{t}
        \|\overline{\varphi}_{h}(s) - \overline{\varphi}_{\tau}(s) \|_{L^2(\Omega)}^2\,ds   
\notag \\[3mm]
&= 3\int_{0}^{t}\|\overline{\varphi}_{h}(s) - \widehat{\varphi}_{h}(s) 
               + \widehat{\varphi}_{\tau}(s) - \overline{\varphi}_{\tau}(s) 
               + \widehat{\varphi}_{h}(s) - \widehat{\varphi}_{\tau}(s)\|_{L^2(\Omega)}^2\,ds 
\notag \\ 
&\leq 9\int_{0}^{t}
            \|\overline{\varphi}_{h}(s) - \widehat{\varphi}_{h}(s)\|_{L^2(\Omega)}^2\,ds   
       + 9\int_{0}^{t}
            \|\widehat{\varphi}_{\tau}(s) - \overline{\varphi}_{\tau}(s)\|_{L^2(\Omega)}^2\,ds 
\notag \\   
    &\,\quad 
        + 9\int_{0}^{t}
               \|\widehat{\varphi}_{h}(s) - \widehat{\varphi}_{\tau}(s)\|_{L^2(\Omega)}^2\,ds  
\notag \\ 
&\leq C_{2}h^2 + C_{2}\tau^2  
+ 9\int_{0}^{t}
       \|\widehat{\varphi}_{h}(s) - \widehat{\varphi}_{\tau}(s)\|_{L^2(\Omega)}^2\,ds  
\end{align}
for all $h, \tau \in (0, h_{3})$ and all $t \in [0, T]$. 
We derive from the identity $\overline{v}_{h}(s) = (\widehat{\varphi}_{h})_{s}(s)$, 
\eqref{tool6} and Lemma \ref{esth8} that 
there exists a constant $C_{3} > 0$ such that 
\begin{align}\label{i3}
&\|\widehat{\varphi}_{h}(t) - \widehat{\varphi}_{\tau}(t)\|_{L^2(\Omega)}^2  
\notag \\[3mm] 
&= \left\|
        \int_{0}^{t} (\overline{v}_{h}(s) - \overline{v}_{\tau}(s))\,ds
     \right\|_{L^2(\Omega)}^2
\notag \\ 
&= \left\|
        \int_{0}^{t} (\overline{v}_{h}(s) - \widehat{v}_{h}(s) 
                        + \widehat{v}_{\tau}(s) - \overline{v}_{\tau}(s)
                        + \widehat{v}_{h}(s) - \widehat{v}_{\tau}(s))\,ds
     \right\|_{L^2(\Omega)}^2  
\notag \\ 
&\leq C_{3}h^2 + C_{3}\tau^2 
         + C_{3}\int_{0}^{t}\|\widehat{v}_{h}(s) - \widehat{v}_{\tau}(s)\|_{L^2(\Omega)}^2\,ds 
\end{align}
for all $h, \tau \in (0, h_{3})$ and all $t \in [0, T]$. 
Thus, since 
\begin{align*}
&\widehat{v}_{h} - \widehat{v}_{\tau} 
+ \widehat{\varphi}_{h} - \widehat{\varphi}_{\tau} 
+ a(\cdot)(1\star(\underline{\varphi}_{h} - \underline{\varphi}_{\tau})) 
- J\ast(1\star(\underline{\varphi}_{h} - \underline{\varphi}_{\tau})) 
\notag \\ 
&+ 1\star(\beta(\overline{\varphi}_{h}) - \beta(\overline{\varphi}_{\tau})) 
  + 1\star(\pi(\overline{\varphi}_{h}) - \pi(\overline{\varphi}_{\tau})) 
= 1\star(\overline{u}_{h} - \overline{u}_{\tau}), 
\end{align*}
we deduce from {\bf A1}, Lemma \ref{esth5}, 
the local Lipschitz continuity of $\beta$, {\bf A3}, 
\eqref{i1}-\eqref{i3} that 
there exists a constant $C_{4} > 0$ such that 
\begin{align}\label{i4}
&\|\widehat{v}_{h}(t) - \widehat{v}_{\tau}(t)\|_{L^2(\Omega)}^2 
\notag \\[4mm]
&\leq C_{4}(h^2 + \tau^2) 
    + C_{4}\int_{0}^{t}
                \|\widehat{\varphi}_{h}(s) - \widehat{\varphi}_{\tau}(s)\|_{L^2(\Omega)}^2\,ds 
\notag \\   
&\,\quad + C_{4}\int_{0}^{t}
                        \|\widehat{v}_{h}(s) - \widehat{v}_{\tau}(s)\|_{L^2(\Omega)}^2\,ds 
    + C_{4}\|(1\star(\overline{u}_{h} - \overline{u}_{\tau}))(t)\|_{H^1(\Omega)}^2 
\end{align}
for all $h, \tau \in (0, h_{3})$ and all $t \in [0, T]$. 
On the other hand, 
it follows from the identity $\overline{v}_{h}(s) = (\widehat{\varphi}_{h})_{s}(s)$, 
the Schwarz inequality, the Young inequality, 
\eqref{tool6}, 
Lemmas \ref{esth1} and \ref{esth8} 
that there exists a constant $C_{5} > 0$ such that 
\begin{align}\label{i5}
&\frac{1}{2}\|\widehat{\varphi}_{h}(t) - \widehat{\varphi}_{\tau}(t)\|_{L^2(\Omega)}^2 
\notag \\[4mm]
&= \int_{0}^{t}(\overline{v}_{h}(s) - \overline{v}_{\tau}(s), 
                  \widehat{\varphi}_{h}(s) - \widehat{\varphi}_{\tau}(s))_{L^2(\Omega)}\,ds 
\notag \\ 
&= \int_{0}^{t}(\overline{v}_{h}(s) - \widehat{v}_{h}(s), 
                  \widehat{\varphi}_{h}(s) - \widehat{\varphi}_{\tau}(s))_{L^2(\Omega)}\,ds  
\notag \\ 
&\,\quad+ \int_{0}^{t}(\widehat{v}_{\tau}(s) - \overline{v}_{\tau}(s), 
                  \widehat{\varphi}_{h}(s) - \widehat{\varphi}_{\tau}(s))_{L^2(\Omega)}\,ds  
\notag \\ 
&\,\quad + \int_{0}^{t}(\widehat{v}_{h}(s) - \widehat{v}_{\tau}(s), 
                     \widehat{\varphi}_{h}(s) - \widehat{\varphi}_{\tau}(s))_{L^2(\Omega)}\,ds 
\notag \\[4mm] 
&\leq C_{5}h + C_{5}\tau 
   + \frac{1}{2}\int_{0}^{t}\|\widehat{v}_{h}(s) - \widehat{v}_{\tau}(s)\|_{L^2(\Omega)}^2\,ds  
   + \frac{1}{2}
           \int_{0}^{t}
               \|\widehat{\varphi}_{h}(s) - \widehat{\varphi}_{\tau}(s)\|_{L^2(\Omega)}^2\,ds                    
\end{align}
for all $h, \tau \in (0, h_{3})$ and all $t \in [0, T]$. 
Therefore we can show Lemma \ref{LemC2} by \eqref{i4} and \eqref{i5}. 
\end{proof}

\begin{lem}\label{Cauchy}
Let $h_{3}$ be as in Lemma \ref{esth5}. 
Then there exists a constant  $M > 0$ depending on the data such that 
\begin{align*}
&\|1\star(\overline{u}_{h} - \overline{u}_{\tau})\|_{C([0, T]; H^1(\Omega))}
   + \|\widehat{\varphi}_{h} - \widehat{\varphi}_{\tau}\|_{C([0, T]; L^2(\Omega))}
  + \|\widehat{v}_{h} - \widehat{v}_{\tau}\|_{C([0, T]; L^2(\Omega))}
\notag \\[2mm]
&\leq M(h^{1/2} + \tau^{1/2}) 
        + M\|\overline{f}_{h} - \overline{f}_{\tau}\|_{L^2(0, T; L^2(\Omega))} 
           + M\|\overline{g}_{h} - \overline{g}_{\tau}\|_{L^2(0, T; L^2(\partial\Omega))} 
\end{align*}
for all $h, \tau \in (0, h_{3})$. 
\end{lem}
\begin{proof}
Combining \eqref{C1} and \eqref{C2} leads to the inequality 
\begin{align*}
&\frac{1}{2}\|(1\star(\overline{u}_{h} - \overline{u}_{\tau}))(t)\|_{H^1(\Omega)}^2 
+ \frac{1}{2M_{2}}
       \|\widehat{\varphi}_{h}(t) - \widehat{\varphi}_{\tau}(t)\|_{L^2(\Omega)}^2 
+ \frac{1}{2M_{2}}
       \|\widehat{v}_{h}(t) - \widehat{v}_{\tau}(t)\|_{L^2(\Omega)}^2  
\notag \\[5mm] 
&\leq \left(M_{1} + \frac{1}{2}\right)(h + \tau) 
         + M_{1}\int_{0}^{t} 
                 \|(1\star(\overline{u}_{h} - \overline{u}_{\tau}))(s)\|_{H^1(\Omega)}^2\,ds 
\notag \\ 
&\,\quad + \frac{1}{2}\int_{0}^{t} 
                 \|\widehat{\varphi}_{h}(s) - \widehat{\varphi}_{\tau}(s)\|_{L^2(\Omega)}^2\,ds 
+ \left(M_{1} + \frac{1}{2}\right)
                   \int_{0}^{t} \|\widehat{v}_{h}(s) - \widehat{v}_{\tau}(s)\|_{L^2(\Omega)}^2\,ds
\notag \\
&\,\quad + M_{1}\|\overline{f}_{h} - \overline{f}_{\tau}\|_{L^2(0, T; L^2(\Omega))}^2 
           + M_{1}\|\overline{g}_{h} - \overline{g}_{\tau}\|_{L^2(0, T; L^2(\partial\Omega))}^2.  
\end{align*}
Thus by the Gronwall lemma we can obtain Lemma \ref{Cauchy}. 
\end{proof}

\medskip

\begin{prth1.1}
We have from Lemmas \ref{esth1}-\ref{esth3}, \ref{esth5}-\ref{esth9}, \ref{Cauchy}, 
the Aubin--Lions lemma 
for the compact embedding $L^2(\Omega) \hookrightarrow {(H^1(\Omega))}^{*}$, 
the properties \eqref{tool4}-\eqref{tool7} that 
there exist some functions $u$, $\theta$, $\varphi$, $\xi$ such that 
\begin{align*}
    &u \in L^2(0, T; H^1(\Omega)),\   
      \theta \in H^1(0, T; {(H^1(\Omega))}^{*}) \cap L^{\infty}(0, T; L^2(\Omega)),  
    \\[1mm]       
    &\varphi \in W^{2, 2}(0, T; L^2(\Omega)) \cap W^{1, \infty}(0, T; L^{\infty}(\Omega)),\ 
    \xi \in L^{\infty}(\Omega\times(0, T))   
    \end{align*}
and 
\begin{align}
&\widehat{\theta}_{h} \to \theta 
\quad \mbox{weakly$^*$ in}\ H^1(0, T; {(H^1(\Omega))}^*) 
                                                                 \cap L^{\infty}(0, T; L^2(\Omega)), 
\label{weakh1} \\[1.5mm] 
&\widehat{\theta}_{h} \to \theta 
\quad \mbox{strongly in}\ C([0, T]; {(H^1(\Omega))}^*),  
\label{strongh1} \\[1.5mm] 
&\alpha(\overline{u}_{h}) = \overline{\theta}_{h} \to \theta 
\quad \mbox{weakly$^*$ in}\ L^{\infty}(0, T; L^2(\Omega)), 
\label{weakh2} \\[1.5mm] 
&\overline{u}_{h} \to u 
\quad \mbox{weakly in}\ L^2(0, T; H^1(\Omega)), 
\label{weakh3} \\[1.5mm] 
&\overline{z}_{h} \to \varphi_{tt} 
\quad \mbox{weakly in}\ L^2(0, T; L^2(\Omega)),  
\label{weakh4} \\[1.5mm] 
&\widehat{v}_{h} \to \varphi_{t} 
\quad \mbox{strongly in}\ C([0, T]; L^2(\Omega)),   
\label{strongh2} \\[1.5mm] 
&\overline{v}_{h} \to \varphi_{t} 
\quad \mbox{weakly$^*$ in}\  L^{\infty}(\Omega\times(0, T)),   
\label{weakh5} \\[1.5mm] 
&\widehat{\varphi}_{h} \to \varphi 
\quad \mbox{weakly$^*$ in}\ W^{1, \infty}(0, T; L^{\infty}(\Omega)), 
\label{weakh6} \\[1.5mm]  
&\widehat{\varphi}_{h} \to \varphi 
\quad \mbox{strongly in}\ C([0, T]; L^2(\Omega)), 
\label{strongh3} \\[1.5mm]  
&\overline{\varphi}_{h} \to \varphi 
\quad \mbox{weakly$^*$ in}\ L^{\infty}(\Omega\times(0, T)),   
\notag \\[1.5mm] 
&\underline{\varphi}_{h} \to \varphi 
\quad \mbox{weakly$^*$ in}\ L^{\infty}(\Omega\times(0, T)),   
\label{weakh7} \\[1.5mm]
&\beta(\overline{\varphi}_{h}) \to \xi 
\quad \mbox{weakly$^*$ in}\ L^{\infty}(\Omega\times(0, T))  
\label{weakh8}
\end{align}
as $h = h_{j} \searrow 0$, 
where $\alpha(r) := - \frac{1}{r}$ 
for $r \in D(\alpha) := \{ r \in \mathbb{R} \ |\  r < 0 \}$. 
We see from \eqref{tool4}, Lemmas \ref{esth1}, \ref{esth3}, 
\eqref{strongh1}, and \eqref{weakh3} 
that 
\begin{align*}
&\int_{0}^{T} (\alpha(\overline{u}_{h}(t)), \overline{u}_{h}(t))_{L^2(\Omega)}\,dt 
\notag \\ 
&= \int_{0}^{T} (\overline{\theta}_{h}(t), \overline{u}_{h}(t))_{L^2(\Omega)}\,dt 
\notag \\ 
&= \int_{0}^{T} \langle \overline{\theta}_{h}(t) - \widehat{\theta}_{h}(t), 
                            \overline{u}_{h}(t) \rangle_{{(H^1(\Omega))}^*, H^1(\Omega)}\,dt 
     + \int_{0}^{T} \langle \widehat{\theta}_{h}(t), 
                            \overline{u}_{h}(t) \rangle_{{(H^1(\Omega))}^*, H^1(\Omega)}\,dt  
\notag \\ 
&\to \int_{0}^{T} \langle \theta(t), u(t) \rangle_{{(H^1(\Omega))}^*, H^1(\Omega)}\,dt   
= \int_{0}^{T} (\theta(t), u(t))_{L^2(\Omega)}\,dt  
\end{align*}
as $h = h_{j} \searrow 0$. 
Thus, noting that $\alpha : D(\alpha) \subset \mathbb{R} \to \mathbb{R}$ 
is maximal monotone, 
we can obtain that 
\begin{equation}\label{j1}
\theta = \alpha(u) = - \frac{1}{u} \quad \mbox{a.e.\ in}\ \Omega\times (0, T)
\end{equation}
(see, e.g., \cite[Lemma 1.3, p.\ 42]{Barbu1}). 
On the other hand, 
it follows from \eqref{tool5}, Lemma \ref{esth5} and \eqref{strongh3} that 
\begin{align}\label{j2}
\|\overline{\varphi}_{h} - \varphi\|_{L^{\infty}(0, T; L^2(\Omega))} 
&\leq \|\overline{\varphi}_{h} - \widehat{\varphi}_{h}\|_{L^{\infty}(0, T; L^2(\Omega))} 
       + \|\widehat{\varphi}_{h} - \varphi\|_{L^{\infty}(0, T; L^2(\Omega))} 
\notag \\ 
&\leq |\Omega|^{1/2} h \|\overline{v}_{h}\|_{L^{\infty}(\Omega\times(0, T))} 
         + \|\widehat{\varphi}_{h} - \varphi\|_{C([0, T]; L^2(\Omega))} 
\notag \\ 
&\to 0
\end{align}
as $h = h_{j} \searrow 0$. 
Then combining \eqref{weakh8} and \eqref{j2} yields that 
\begin{align*}
\int_{0}^{T}(\beta(\overline{\varphi}_{h}(t)), \overline{\varphi}_{h}(t))_{L^2(\Omega)}\,dt 
\to \int_{0}^{T} (\xi(t), \varphi(t))_{L^2(\Omega)}\,dt
\end{align*} 
as $h=h_{j}\searrow0$, and hence it holds that  
\begin{align}\label{j3}
\xi = \beta(\varphi) \quad \mbox{a.e.\ in}\ \Omega\times(0, T).  
\end{align}  
Therefore by virtue of \eqref{weakh1}, \eqref{strongh1}, 
\eqref{weakh3}-\eqref{j3}, {\bf A1}, {\bf A3} 
and by observing that 
$\overline{f}_{h} \to f$ strongly in $L^2(0, T; L^2(\Omega))$ 
and $\overline{g}_{h} \to g$ strongly in $L^2(0, T; L^2(\partial\Omega))$ 
as $h \searrow 0$ (see e.g., \cite[Section 5]{CK1}),  
we can derive existence of weak solutions to \eqref{P}. 
Moreover, we can show uniqueness of weak solutions to \eqref{P} 
in a similar way to the proofs of Lemmas \ref{LemC1}, \ref{LemC2} and \ref{Cauchy}. 
\qed
\end{prth1.1}

\begin{prth2.2}
Since we have from 
$f \in L^2(0, T; L^2(\Omega)) \cap W^{1,1}(0, T; L^2(\Omega))$ 
and $g \in L^2(0, T; L^2(\partial\Omega)) \cap W^{1,1}(0, T; L^2(\partial\Omega))$
that there exists a constant $C_{1} >0$ such that 
\[
\|\overline{f}_{h} - f\|_{L^2(0, T; L^2(\Omega))} \leq C_{1}h^{1/2}
\]
and 
\[
\|\overline{g}_{h} - g\|_{L^2(0, T; L^2(\partial\Omega))} \leq C_{1}h^{1/2}
\]
for all $h > 0$ (see e.g., \cite[Section 5]{CK1}), 
we can prove Theorem \ref{maintheorem3} by Lemma \ref{Cauchy}.  
\qed
\end{prth2.2}



%
%
%


\begin{thebibliography}{99} 
%
\bibitem{Barbu1}
V. Barbu, 
``Nonlinear Semigroups and Differential Equations in Banach spaces'',  
Noordhoff International Publishing, Leyden, 1976. 
%
\bibitem{Barbu2}
V. Barbu,
``Nonlinear Differential Equations of Monotone Types in Banach Spaces'',
Springer, New York, 2010.
%
%
\bibitem{CGI2002}
P. Colli, M. Grasselli,  A. Ito, 
{\it On a parabolic-hyperbolic Penrose-Fife phase-field system}, 
Electron. J. Differential Equations 2002, No. 100, 30 pp.\ 
(Erratum: Electron. J. Differential Equations 2002, No. 100, 32 pp.).
%
\bibitem{CK1}
P. Colli, S. Kurima, 
    {\it Time discretization of a nonlinear phase field system 
in general domains}, Comm.\ Pure Appl.\ Anal.\ {\bf 18} (2019), 3161--3179.   
%
\bibitem{Jerome}
J.W.\ Jerome, 
``Approximations of Nonlinear Evolution Systems'',
Mathematics in Science and Engineering {\bf 164},
Academic Press Inc., Orlando, 1983. 
%
\bibitem{K7}
S. Kurima, 
    {\it Time discretization of a nonlocal phase-field system with inertial term}, 
Matematiche (Catania), to appear.   
%
\bibitem{K8}
S. Kurima, 
{\it Existence for a singular nonlocal phase field system with inertial term}, 
Acta Appl. Math., to appear. 
%
%
%
\end{thebibliography}
\end{document}